\documentclass[11pt,letterpaper]{amsart}
\usepackage{
  amssymb, xcolor, mathtools, mleftright, xurl}
\usepackage[shortlabels]{enumitem}
\usepackage{hyperref}

\usepackage[margin=1in]{geometry}

\newtheorem{theorem}{Theorem}[section]
\newtheorem{lemma}[theorem]{Lemma}

\newtheorem{prop}[theorem]{Proposition}

\theoremstyle{definition}

\newtheorem{defn}[theorem]{Definition}

\theoremstyle{remark}

\DeclareMathOperator{\ct}{ct}
\DeclareMathOperator{\ctidl}{\mathfrak{ct}}

\DeclareMathOperator{\disc}{disc}
\DeclareMathOperator{\Ht}{Ht}
\DeclareMathOperator{\Htp}{Htp}
\DeclareMathOperator{\ind}{ind}
\DeclareMathOperator{\sgn}{sgn}
\newcommand{\monic}{\mathrm{monic}}
\newcommand{\prim}{\mathrm{prim}}
\DeclareMathOperator{\Res}{Res}

\DeclareMathOperator{\Disc}{Disc}

\DeclareMathOperator{\Cl}{Cl}

\newcommand{\FF}{\mathbb{F}}
\newcommand{\QQ}{\mathbb{Q}}
\newcommand{\PP}{\mathbb{P}}
\newcommand{\RR}{\mathbb{R}}
\newcommand{\ZZ}{\mathbb{Z}}

\renewcommand{\aa}{\mathfrak{a}}
\newcommand{\bb}{\mathfrak{b}}

\newcommand{\D}{\mathcal{D}}
\newcommand{\E}{\mathcal{E}}

\newcommand{\OO}{\mathcal{O}}

\newcommand{\cross}{\times}

\newcommand{\textand}{\quad \text{and} \quad}
\renewcommand{\to}{\mathop{\rightarrow}\limits}
\newcommand{\size}[1]{\lvert #1 \rvert}

\newcommand{\floor}[1]{\left\lfloor #1 \right\rfloor}
\newcommand{\ceil}[1]{\left\lceil #1 \right\rceil}
\newcommand{\1}{\mathbf{1}}
\newcommand{\0}{\mathbf{0}}
\newcommand{\intsec}{\cap}

\newcommand{\<}{\left\langle}
\renewcommand{\>}{\right\rangle}

\renewcommand{\epsilon}{\varepsilon}

\mleftright

\title[Galois groups of twisted reciprocal polynomials]{Galois groups of twisted reciprocal polynomials:\\a uniform asymptotic}
\author{Evan M. O'Dorney}

\begin{document}

\begin{abstract}
We study the Galois group $G_f$ of a random polynomial $f$ in the family of polynomials of degree $2n$ satisfying the twisted reciprocal relation $f(x) = x^{2n}/b^n \cdot f(b/x)$.  We use a Euclidean height adapted to this relation.  Our main result is an asymptotic theorem of van der Waerden--Bhargava type: for fixed $b \neq 0$ and $n \geq 4$, the number of polynomials of height at most $H$ whose Galois group is not the full hyperoctahedral group $S_2 \wr S_n$ is an explicit constant times $H^n\log H$, with an error of order $H^n$.  We determine the dependence of the leading constant on $b$; the leading-order group $G_1$ is of index $2$. This paper is a sequel to a recent paper by Anderson, Bertelli, and the author addressing reciprocal polynomials (i.e.\ the case $b = 1$).
\end{abstract}


\maketitle

\section{Introduction}

A polynomial $f(x) = \sum_{i = 0}^n a_i x^i$ of fixed degree $n$ with random integer coefficients $a_i$ in an expanding range $[-H, H]$ typically has Galois group $G_f$ the full symmetric group $S_n$, by Hilbert's irreducibility theorem. In 1936, van der Waerden \cite{vdW1936} studied the incidence of other Galois groups and conjectured that the probability that $G_f \neq S_n$ is bounded above and below by constant multiples of $1/H$, the leading-order contribution coming from the Galois group $G_f = S_{n-1}$ of polynomials with a rational root. Van der Waerden's conjecture was recently resolved by Bhargava \cite{Bhargava_vdW_short,Bhargava_vdW}, as the capstone of eight decades of steady improvements \cite{Knobloch1955,Gallagher1972,CD2020,6author}. Bhargava's approach harnesses sophisticated known results (classification of subgroups of $S_n$, distribution of discriminants of number fields) in combination with innovative recent methods, including Fourier equidistribution and the use of the double discriminant $\disc_{a_0} \disc_x f$.

In a recent paper \cite{ABO_RecipPolys}, Anderson, Bertelli, and the author took on the analogous question for \emph{reciprocal polynomials,} that is, those $f$ of degree $2n$ satisfying the palindromic relation $a_{n+i} = a_{n-i}$. This causes the roots of $f$ to lie in reciprocal pairs $\{\alpha, 1/\alpha\}$, so the typical Galois group is now the hyperoctahedral group $G_0 = S_2\wr S_n$, the subgroup of $S_{2n}$ preserving the partition into $n$ pairs. We found that the techniques developed by Bhargava and his predecessors could also be adapted to resolve this case, answering a question of Davis--Duke--Sun \cite{DDS1998}, with two unexpected differences:
\begin{enumerate}[(i)]
  \item\label{intro:G1} The leading-order contribution is \emph{not} from a non-transitive subgroup such as $S_2 \wr (S_{n-1} \cross S_1)$, but rather from the index-$2$ subgroup $G_1 = (S_2 \wr S_n) \intsec A_{2n}$;
  \item\label{intro:log} The probability of a Galois group $G_f \subsetneq G_0$ is not $\Theta(H^{-1})$ but $\Theta(H^{-1} \log H)$, the log factor arising naturally from an appearance of the harmonic series while counting $G_1$-polynomials.
  \item\label{intro:better} All other Galois groups provably contribute a strictly lower-order count of polynomials. This situation contrasts favorably with the state of the art for general (nonreciprocal) degree-$n$ polynomials, where Bhargava's unconditional bounds still allow transitive groups (such as $A_n$) to appear with frequency $O(H^{-1})$, potentially tying the dominant reducible contribution, although the expected frequency is only $H^{-n/2+o(1)}$; see \cite{MR4768827,AndersonEtAlStrongVdW}.
\end{enumerate}
In this sequel to \cite{ABO_RecipPolys}, we sharpen the surprising findings of \ \ref{intro:G1} and \ref{intro:log} above to an asymptotic and also extend them to the \emph{twisted reciprocal} case, that is, polynomials whose roots lie in pairs $\{\alpha, b/\alpha\}$ for fixed $b$. We expect that Bhargava's techniques will apply to many other linear families of polynomials.  We state our results precisely below.

\subsection{\texorpdfstring{$b$}{b}-Reciprocal polynomials}
Let $n \geq 1$ and $b \neq 0$ be integers. We say that a polynomial $f(x) = \sum_{i=0}^{2n} a_i x^i \in \RR[x]$ of degree $2n$ is \emph{$b$-reciprocal} if it satisfies the linear relations
\begin{equation}
  a_{n-i} = b^i a_{n+i}, \quad i = 0, 1,\ldots, n,
\end{equation}
or, in other words, the identity
\[
  f(x) = \frac{x^{2n}}{b^n} f\left(\frac{b}{x}\right).
\]
If $f \in \ZZ[x]$ and $\disc f \neq 0$, then we can consider the Galois group $G_f \subseteq S_{2n}$ of the roots of $f$. Note that if $\alpha \neq 0$ is a root of $f$, so is $b/\alpha$, implying that $G_f$ is a subgroup of the hyperoctahedral group $G_0 = S_2 \wr S_n$. We are interested in the question of how often $G_f \subsetneq G_0$.

These polynomials are important to study due to their appearance as characteristic polynomials. As already mentioned, any matrix from the symplectic group $\mathrm{Sp}({2n}, \Bbbk)$ has reciprocal characteristic polynomial. Similarly, any matrix from the symplectic similitude group $\mathrm{GSp}(2n, \Bbbk)$ has $b$-reciprocal characteristic polynomial, where $b$ is the factor by which it scales the fixed symplectic form. It is desirable to know how often a $b$-reciprocal polynomial appears as a characteristic polynomial in this group. Over finite fields, this was done by Chavdarov \cite{Chavdarov}; this paper sheds possible light on the corresponding question for symplectic similitudes over $\ZZ$. 

While it is possible to derive results ordering such polynomials by the na\"ive height $\max |a_i|$, it is more pleasant to weight the coefficients in a manner that respects the $b$-symmetry:
\begin{equation}\label{eq:b_height}
  \Ht_b(f)=\sqrt{\sum_{i=0}^{2n}|b|^i a_i^2},
\end{equation}
equivalently
\begin{equation}\label{eq:b_height_free}
  \Ht_b(f)^2=|b|^n a_n^2+2\sum_{i=0}^{n-1} |b|^i a_i^2.
\end{equation}
Note that the real polynomials of height at most $H$ form an ellipsoid in the $n+1$ free coefficients, and if $f$ is an integer polynomial with $a_{2n} \neq 0$ we have $\Ht_b(f) \geq \sqrt{2}|b|^n$.

When $b = 1$, these $f$ are \emph{reciprocal polynomials}, which naturally arise as characteristic polynomials of symplectic matrices, and their Galois groups were studied in \cite{ABO_RecipPolys}.  Here we ask the corresponding question for the whole $b$-reciprocal family.  Of particular note is the case when $b=-1$: these $f$ are \emph{skew-reciprocal polynomials}.  Skew-reciprocal polynomials arise as characteristic polynomials of antisymplectic matrices, and have numerous applications \cite{Hua}.
\begin{theorem} \label{thm:main}
  Let $n \geq 4$ and $b \neq 0$ be fixed integers. Let $\E_{n,b}(H)$ be the number of separable $b$-reciprocal polynomials $f$ of degree $2n$ with height $\Ht_b(f)\leq H$ whose Galois group is not $S_2 \wr S_n$, and let $\E_{n,b}^\monic$ count the subset of these that are monic. Then
  \begin{align}
    \E_{n,b}(H)
    &=\kappa_{n,\sgn b}\,\rho(|b|)\,|b|^{-n(3n+1)/4}H^n\log H
    +O_{n,b}(H^n) \label{eq:main} \\
    \E_{n,b}^\monic(H)
    &=\kappa_{n-1,\sgn b}\,\rho(|b|)\, |b|^{-3n(n-1)/4}H^{n-1}\log H
    +O_{n,b}(H^{n-1}), \label{eq:monic}
  \end{align}
  where the archimedean constant
  \begin{equation}\label{eq:kappa_nonmonic}
    \kappa_{n,\iota}
    =\frac{V_n}{2^{n/2}\zeta(2)\sqrt{2\ceil{n/2}}}
    K\left(\frac{1-\iota}{2}
    -\frac{2\floor{n/2}+1}{2\ceil{n/2}}\right) > 0
  \end{equation}
  is given in terms of the volume $V_n=\pi^{n/2}/\Gamma(n/2+1)$ of the unit $n$-ball and Legendre's complete elliptic integral
  \[
    K(m)=\int_0^{\pi/2}\frac{d\phi}{\sqrt{1-m\sin^2\phi}}
    \qquad(m<1),
  \]
  while the nonarchimedean constant is
  \begin{equation}\label{eq:rho_b}
    \rho(|b|)=\prod_{p^e\parallel \size{b}}
    \left(1+e\frac{p-1}{p+1}\right)
  \end{equation}
\end{theorem}

In fact, more can be said: $100\%$ of the contribution to $\E_{n,b}(H)$ as $H \to \infty$ comes from a single Galois group, denoted $G_1$ below; all other Galois groups appear at most $O_{n,b}(H^n)$ times (and likewise in the monic case). This can be seen from our bounds and the group theory of \cite[Remark 3.6]{ABO_RecipPolys}, which shows that all other subgroups of $G_0 = S_2 \wr S_n$ are contained in $G_2$, $G_3$, or $S_2 \wr H$ for some $H \subsetneq S_n$.

Our result reveals a possible connection between the sets
\[
\{f:\deg{f} = 2n, f(x) = g(x)h(x), \deg{g} = \deg{h} = n\}
\]
and 
\[
\{f: \deg{f} = 2n, f \text{ is $b$-reciprocal, } G_f \neq S_2\wr S_n\};
\]
indeed van der Waerden showed \cite{vdW1936} that for coefficients in $[-H,H]$, polynomials in the first set occur with frequency $H^{-n}\log{H}$, which matches our count for the second set.

For comparison, the total number of $b$-reciprocal polynomials in Theorem \ref{thm:main} is
\[
  \frac{V_{n+1}}{2^{n/2}|b|^{3n(n+1)/4}}H^{n+1}+O_{n,b}(H^n),
\]
where $V_m$ is the volume of the Euclidean unit ball in $\RR^m$.  Thus the exceptional probability is $O_{n,b}(\log H/H)$, as in \cite{ABO_RecipPolys}.

\subsection{Notation}
Throughout this paper, $n$ and $b$ are fixed, while $H \to \infty$. We carry over the same notational conventions as in \cite{ABO_RecipPolys}, as far as practical.

If $g(x)=a_0+a_1x+\cdots+a_mx^m$ is a polynomial over a number field $K$, we define its \emph{content ideal}
\[
  \ctidl(g) = (a_0, a_1,\ldots, a_m)
\]
to be the fractional ideal generated by the coefficients. In the case $K = \QQ$, $\ctidl(g) = \bigl(\ct(g)\bigr)$ is principal, with a unique positive generator which we call the \emph{content} of $g$.

If $g(x)=a_0+a_1x+\cdots+a_mx^m$ is an integer polynomial, define its \emph{height} by the Euclidean norm of its coefficient vector,
\[
  \Ht(g)=\left(\sum_{i=0}^m a_i^2\right)^{1/2}.
\]
If $g\neq0$, define its \emph{projective height} by
\[
  \Htp(g)=\frac{\Ht(g)}{\ct(g)}.
\]
Thus $\Htp(g)=\Ht(g)$ when $g$ is primitive. For polynomials of degree at most $n$, these Euclidean heights differ by at most a constant factor depending only on $n$ from the $\infty$-norm heights used in \cite{ABO_RecipPolys}. Consequently, all estimates from that paper stated in terms of $\Ht$ or $\Htp$ may be used here without change, although the implied constants may need to be modified by factors depending only on $n$.

\subsection{Acknowledgments}
I thank Theresa ``Tess'' Anderson for useful discussions, original ideas, and writing that went into an earlier draft of this article,  \href{https://arxiv.org/abs/2603.15875}{arxiv:2603.15875} (\textsection2). As this version is heavily AI-generated and the results are noticeably stronger, I have decided to post it as a new arXiv article.

\section{\texorpdfstring{$b$}{b}-Reciprocal polynomials}
If $f$ is a $b$-reciprocal polynomial of degree $2n$, it is easy to see that there is a unique degree $n$ integer-coefficient polynomial
\[
  g(u) = c_n u^n + \cdots + c_1 u + c_0
\]
such that
\[
  f(x) = x^n g\left(x + \frac{b}{x}\right).
\]
The passage between $f$ and $g$ is bijective and linear.  For fixed $n$ and $b$, the height $\Ht_b(f)$ is equivalent as a norm to any fixed norm on the coefficients of $g$.  Thus the upper bounds from \cite{ABO_RecipPolys} and \cite{Bhargava_vdW} continue to apply, with implied constants depending on $n$ and $b$.  For the leading constant, however, we retain the precise ellipsoid defined by the height \eqref{eq:b_height}.

We denote the roots of $f$ by
\[
  \alpha_1, \frac{b}{\alpha_1}, \alpha_2, \frac{b}{\alpha_2}, \ldots, \alpha_n, \frac{b}{\alpha_n}.
\]
\[
  \beta_i = \alpha_i + \frac{b}{\alpha_i}.
\]
We will sometimes write $\alpha = \alpha_1$ and $\beta = \beta_1$ when the choice of root is irrelevant.

\begin{lemma} \label{lem:disc_f}
  $\disc f = b^{n(n-1)} g\bigl(2\sqrt{b}\bigr) g\bigl(-2\sqrt{b}\bigr) (\disc g)^2$.
\end{lemma}
\begin{proof}
  Express both sides in terms of the $\alpha_i$ as in \cite[Lemma 2.1]{ABO_RecipPolys}.
\end{proof} 

Assuming $n \geq 4$, we may discard any set containing $O_{n,b}(H^n)$ polynomials in the nonmonic family, or $O_{n,b}(H^{n-1})$ in the monic family, without affecting the leading term.  In particular, by the results of \cite{Bhargava_vdW} and the adaptations in \cite{ABO_RecipPolys}, we may assume that $g$ is irreducible, that $G_g=S_n$, and that $g(\pm 2\sqrt b)$ are nonzero, in particular that $f$ is separable.

We have the tower of number fields
  \[
    K_f = \QQ(\alpha) \quad \supseteq \quad K_g = \QQ(\beta) \quad \supset \quad \QQ,
  \]
  where $K_g$ is an extension of degree $n$, while $K_f/K_g$ is of degree at most $2$, given by $K_f = K_g(\sqrt{\beta^2 - 4 b})$. Let $\widetilde{K}_g$ and $\widetilde{K}_f$, respectively, be the splitting fields of $g$ and $f$, and let $G_g$ and $G_f$ be their respective Galois groups, which are subgroups of $S_n$ and $G_0$, with $G_f \twoheadrightarrow G_g$ under the natural projection $G_0 \twoheadrightarrow S_n$. By the main result of Bhargava \cite[Theorem 1]{Bhargava_vdW}, we can assume that $G_g$ is the whole $S_n$. Our aim in this paper is to understand when $G_f$ is not the whole $G_0$.

  Group-theoretically, this problem behaves just as in \cite{ABO_RecipPolys}:
  
  \begin{prop}[\cite{ABO_RecipPolys}, Theorem 3.4] \label{prop:G_i}
    The maximal subgroups of $G_0$ whose projection onto $S_n$ is the whole group are
    \begin{alignat*}{2}
      G_1 &= \left\{(\mathbf{v}, \sigma) \in \FF_2^n \rtimes S_n : \sum_i v_i = 0\right\} = \<\1\>^\perp \rtimes S_n, &\quad n &\geq 1 \\
      G_2 &= \left\{(\mathbf{v}, \sigma) \in \FF_2^n \rtimes S_n : \sum_i v_i = \sgn \sigma\right\}, &\quad n &\geq 2 \\
      G_3 &= \<\1\> \cross S_n, &\quad n &\geq 3 \text{ odd.}
    \end{alignat*}
  \end{prop}

  To prove Theorem \ref{thm:main}, we must bound the number of polynomials $f$, equivalently $g$, for which $G_f$ is a subgroup of $G_1$, $G_2$, or $G_3$ for the values of $n$ listed in Proposition \ref{prop:G_i}.
  
  \begin{defn} \label{def:E_n}
    For $G \subseteq G_0$ and $H \geq 2$, let $\E_{n,b}(G; H)$ be the number of separable $b$-reciprocal polynomials $f$ of degree $2n$ with $\Ht_b(f)\leq H$ such that $G_g = S_n$ and $G_f \subseteq G$. Let $\E_{n,b}^{\monic}(G; H)$ count the subset of these that are monic.
  \end{defn}

  Applying the method of proof of Lemma 3.8 of \cite{ABO_RecipPolys} to Lemma \ref{lem:disc_f}, we get the following criteria:
  \begin{lemma}\label{lem:G123_conds} Assume that $G_g$ is the whole of $S_n$. Then:
    \begin{enumerate}[$($a$)$]
      \item\label{it:G1} $G_f \subseteq G_1$ if and only if $g\bigl(2\sqrt{b}\bigr) g\bigl(-2\sqrt{b}\bigr)$ is a square.
      \item\label{it:G2} $G_f \subseteq G_2$ if and only if $g\bigl(2\sqrt{b}\bigr) g\bigl(-2\sqrt{b}\bigr) \disc g$ is a square.
    \end{enumerate}
  \end{lemma}

\subsection{Counting \texorpdfstring{$G_1$}{G1}-polynomials}
From $G_1$ we get the main terms of \eqref{eq:main} and \eqref{eq:monic} in Theorem \ref{thm:main}.

\begin{theorem} \label{thm:G1}
  For fixed $b\neq0$ and $n\geq2$, with $\iota=\sgn(b)$, we have
  \begin{align}
    \E_{n,b}(G_1;H)
      &=\kappa_{n,\iota}\rho(|b|)|b|^{-n(3n+1)/4}
        H^n\log H + O_{n,b}(H^n), \label{eq:G1_nonmonic_asymptotic}\\
    \E_{n,b}^{\monic}(G_1;H)
      &=\kappa_{n-1,\iota}\rho(|b|)|b|^{-3n(n-1)/4}
        H^{n-1}\log H + O_{n,b}(H^{n-1}).
        \label{eq:G1_monic_asymptotic}
  \end{align}
\end{theorem}
\begin{proof}
For $g(z)=\sum_{j=0}^n c_j z^j$, write
\[
  g(\pm2\sqrt b)=X\pm2Y\sqrt b,
\]
where
\[
  X=\sum_{j=0}^{\floor{n/2}}c_{2j}(4b)^j
  \textand
  Y=\sum_{j=0}^{\ceil{n/2} - 1}c_{2j+1}(4b)^j.
\]
By Lemma \ref{lem:G123_conds}\ref{it:G1}, the condition for $G_f \subseteq G_1$ is
\begin{equation}\label{eq:conic}
  X^2-4bY^2=Z^2. 
\end{equation}
A parametrization of the $\PP^2(\QQ)$-points of this conic is given by
\[
  [X : Y : Z] = [s^2+bt^2 : st : s^2-bt^2], \quad [s : t] \in \PP^1(\QQ).
\]
Over $\ZZ$, we must account for scaling.  For a primitive pair $(s,t)$, put
\[
  \mathbf v_b(s,t)=(s^2+bt^2,st).
\]
Since
\[
  \gcd(s^2+bt^2,st,s^2-bt^2)=\gcd(s,b),
\]
the nonzero integer solutions of
\eqref{eq:conic} are parametrized by
\begin{equation}\label{eq:param_conic}
  (X,Y,Z)=\frac{\lambda}{\gcd(s,b)}
  (s^2+bt^2,st,s^2-bt^2),
\end{equation}
where $\lambda\in\ZZ\setminus\{0\}$ and $\gcd(s,t)=1$.  Away from
$Z=0$, which contributes only $O_{n,b}(H^n)$ polynomials, a given pair
$(X,Y)$ occurs four times in this parametrization, corresponding to the
symmetries $Z\mapsto-Z$ and $(s,t)\mapsto(-s,-t)$.

Observe that the polynomial $g$ is determined by $X$, $Y$, and its $n-1$ upper coefficients: for $(X,Y)\in\RR^2$ and $\mathbf c=(c_2,\ldots,c_n)\in\RR^{n-1}$, let
\[
  c_0=X-\sum_{j=1}^{\floor{n/2}}c_{2j}(4b)^j,
  \qquad
  c_1=Y-\sum_{j=1}^{\ceil{n/2}-1}c_{2j+1}(4b)^j,
\]
be the coefficients that make
\[
  g_{X,Y,\mathbf c}(z)=\sum_{j=0}^n c_jz^j
\]
have these values for $X$ and $Y$, and let
\[
  f_{X,Y,\mathbf c}(x)=x^n g_{X,Y,\mathbf c}\left(x+\frac bx\right)
\]
be the corresponding reciprocal polynomial. Let
\begin{equation}\label{eq:psi_def}
  \psi_{n,b}(X,Y)
  =\operatorname{vol}_{n-1}\left\{(c_2,\ldots,c_n)\in\RR^{n-1}:
  \Ht_b(f_{X,Y,\mathbf c})\leq1\right\}
\end{equation}
be the volume in the independent coordinates $c_2,\ldots,c_n$ for which $f$ lies within the unit-height ellipsoid $\Ht_b(f)\leq1$. Equivalently, since the linear change of coordinates
\[
  (c_0,\ldots,c_n)\longmapsto(X,Y,c_2,\ldots,c_n)
\]
has determinant $1$, the function $\psi_{n,b}$ is the density with
respect to $dX\,dY$ of the pushforward of Lebesgue measure on the unit-height ellipsoid under the map $g\mapsto(X,Y)$.
In particular, $\psi_{n,b}$ is compactly supported and even.

Applying Davenport's lemma in the affine fibers gives, uniformly
for integer $(X,Y)$ in the image of the ellipsoid $\Ht_b(f)\leq H$,
\begin{equation}\label{eq:fiber_count}
  \#\left\{(c_2,\ldots,c_n)\in\ZZ^{n-1}:
  \Ht_b(f_{X,Y,\mathbf c})\leq H\right\}
  =H^{n-1}\psi_{n,b}(X/H,Y/H)+O_{n,b}(H^{n-2}).
\end{equation}
Thus \eqref{eq:param_conic}, including its fourfold multiplicity, gives
\begin{equation}\label{eq:conic_double_sum}
  \E_{n,b}(G_1; H)
  =\frac{H^{n-1}}4
  \sum_{\substack{(s,t)\in\ZZ^2\\(s,t)=1}}
  \sum_{\lambda\in\ZZ\setminus\{0\}}
  \psi_{n,b}\left(\frac{\lambda}{H \gcd(s,b)}\mathbf v_b(s,t)\right)
  +O_{n,b}(H^n).
\end{equation}
Indeed, the height bound forces
$\|(s,t)\|\ll_{n,b}\sqrt H$ and
$|\lambda|\ll_{n,b}H/\|(s,t)\|^2$.  Hence there are
$O_{n,b}(H\log H)$ relevant parameter triples, so the accumulated error
from \eqref{eq:fiber_count} is $O_{n,b}(H^{n-1}\log H)$.

Let $\mathcal S_H$ be the set of primitive pairs $(s,t)$ for which the
inner sum in \eqref{eq:conic_double_sum} is nonzero. The support of
$\psi_{n,b}$ is compact (it is the projection of the unit-height ellipsoid to the
$(X,Y)$-plane, which is an ellipse centered at the origin). Noting that $\mathbf v_b(s,t) \neq 0$ and
$1\leq\gcd(s,b)\leq |b|$, there are constants $0<c<C$, depending only
on $n$ and $b$, such that
\begin{equation}\label{eq:SH_sandwich}
  \{(s,t)\in\ZZ^2:(s,t)=1,\ 0<\|(s,t)\|\leq c\sqrt H\}
  \subseteq \mathcal S_H
  \subseteq
  \{(s,t)\in\ZZ^2:(s,t)=1,\ 0<\|(s,t)\|\leq C\sqrt H\}.
\end{equation}
For the first inclusion one may take $\lambda=1$; the second follows
from compactness of the support.

For a fixed $(s,t)\in\mathcal S_H$, the inner sum in \eqref{eq:conic_double_sum} is a one-dimensional
Riemann sum.  Thus
\begin{equation}\label{eq:lambda_rsum}
  \sum_{\lambda\in\ZZ\setminus\{0\}}
  \psi_{n,b}\left(\frac{\lambda}{H \gcd(s,b)}\mathbf v_b(s,t)\right)
  =H \gcd(s,b)\int_{\RR}
  \psi_{n,b}\bigl(\lambda\mathbf v_b(s,t)\bigr)\,d\lambda
  +O_{n,b}(1).
\end{equation}
Since $|\mathcal S_H|=O_{n,b}(H)$, summing this error contributes
$O_{n,b}(H^n)$ to \eqref{eq:conic_double_sum}.  Therefore, on putting
\[
  F_{n,b}(s,t)
  =\int_{\RR}\psi_{n,b}\bigl(\lambda\mathbf v_b(s,t)\bigr)\,d\lambda,
\]
we obtain
\begin{equation}\label{eq:conic_reduced}
  \E_{n,b}(G_1;H)
  =\frac{H^n}{4}
  \sum_{(s,t)\in\mathcal S_H}\gcd(s,b)F_{n,b}(s,t)
  +O_{n,b}(H^n).
\end{equation}
Notice that $F_{n,b}$ is continuous and
homogeneous of degree $-2$ in $(s,t)$.

We now use the following standard weighted form of equidistribution of primitive lattice points; its proof is included in Appendix~\ref{app:primitive_lattice}.

\begin{lemma}\label{lem:weighted_primitive_lattice}
Let $q\geq1$, and put
\[
  P_q=\{(u,v)\in(\ZZ/q\ZZ)^2:\gcd(u,v,q)=1\}.
\]
Let $\Phi:P_q\to\RR$, and write
\[
  \langle\Phi\rangle_q
  =\frac{1}{|P_q|}\sum_{(u,v)\in P_q}\Phi(u,v).
\]
Suppose that $W:\RR^2\to\RR$ is bounded and compactly supported, and is
$C^1$ away from a finite union of (necessarily compact) piecewise $C^1$ curves.  Then,
as $T\to\infty$,
\begin{equation}\label{eq:weighted_primitive_lattice}
  \sum_{\substack{(s,t)\in\ZZ^2\\(s,t)=1}}
  \Phi(\bar s,\bar t)W(s/T,t/T)
  =\frac{T^2}{\zeta(2)}\langle\Phi\rangle_q
    \int_{\RR^2}W(s,t)\,ds\,dt
    +O_{q,\Phi,W}(T\log(2T)),
\end{equation}
where $(\bar s,\bar t)$ denotes reduction modulo $q$.
\end{lemma}

Take
$q=|b|$ and
\[
  \Phi_b(u,v)=\gcd(u,b),\qquad (u,v)\in P_{|b|}.
\]
Apply Lemma~\ref{lem:weighted_primitive_lattice}
to
\[
  W_{n,b}(s,t)
  =\mathbf 1_{\{1<\sqrt{s^2+t^2}\leq2\}}F_{n,b}(s,t).
\]
By homogeneity, for $T\geq1$ this gives
\begin{align}\label{eq:dyadic_primitive_sum}
  &\sum_{\substack{(s,t)\in\ZZ^2,\ (s,t)=1\\
          T<\|(s,t)\|\leq2T}}
  \gcd(s,b)F_{n,b}(s,t) \notag\\
  &\qquad=\frac{\<\Phi\>_{|b|}}{\zeta(2)}
    \int_{1<\sqrt{s^2+t^2}\leq2}F_{n,b}(s,t)\,ds\,dt
    +O_{n,b}\left(\frac{\log(2T)}{T}\right).
\end{align}
In polar coordinates the integral here is
\[
  \log 2\int_0^{2\pi}F_{n,b}(\cos\theta,\sin\theta)\,d\theta.
\]
Summing \eqref{eq:dyadic_primitive_sum} over dyadic annuli, and absorbing
the first and last partial annuli into $O_{n,b}(1)$, yields
\begin{equation}\label{eq:weighted_primitive_sum}
  \sum_{\substack{(s,t)\in\ZZ^2,\ (s,t)=1\\0<\|(s,t)\|\leq R}}
  \gcd(s,b)F_{n,b}(s,t)
  =\frac{\<\Phi\>_{|b|}}{\zeta(2)}\log R
  \int_0^{2\pi}F_{n,b}(\cos\theta,\sin\theta)\,d\theta
  +O_{n,b}(1).
\end{equation}

This explains why the archimedean factor naturally takes the form of a
2D real integral.  Define
\begin{equation}\label{eq:arch_factor}
  \mathcal A_{n,b}
  =\int_{\RR^2}\psi_{n,b}(s^2+bt^2,st)\,ds\,dt.
\end{equation}
If $(s,t)=r(\cos\theta,\sin\theta)$ with $r>0$, then
$\mathbf v_b(s,t)=r^2\mathbf v_b(\cos\theta,\sin\theta)$.  Therefore,
putting $u=r^2$ and using the evenness of $\psi_{n,b}$,
\begin{align}\label{eq:arch_factor_polar}
  \mathcal A_{n,b}
  &=\frac12\int_0^{2\pi}\int_0^\infty
    \psi_{n,b}\bigl(u\mathbf v_b(\cos\theta,\sin\theta)\bigr)
    \,du\,d\theta \notag\\
  &=\frac14\int_0^{2\pi}F_{n,b}(\cos\theta,\sin\theta)\,d\theta.
\end{align}
Consequently \eqref{eq:weighted_primitive_sum} becomes
\begin{equation}\label{eq:weighted_primitive_sum_A}
  \sum_{\substack{(s,t)\in\ZZ^2,\ (s,t)=1\\0<\|(s,t)\|\leq R}}
  \gcd(s,b)F_{n,b}(s,t)
  =\frac{4\<\Phi\>_{|b|}}{\zeta(2)}\mathcal A_{n,b}\log R+O_{n,b}(1).
\end{equation}

To evaluate \eqref{eq:arch_factor}, we first eliminate the dependence on $|b|$. Observe that
\[
  \widetilde{f}(y) = f(|b|^{1/2}y) = y^n\sum_{j=0}^n w_j(y+\iota y^{-1})^j
\]
is $\iota$-reciprocal, where $\iota=\sgn(b)$ and $w_j=|b|^{(n+j)/2}c_j$. Now
\begin{equation}\label{eq:height_Q}
  \Ht_b(f)^2 = \Ht_\iota(\widetilde{f})^2 = Q_{n,\iota}(w_0,\ldots,w_n)
\end{equation}
is given by a fixed quadratic form
\begin{equation}\label{eq:Q_nh}
  Q_{n,\iota}(\mathbf w)
  =\left\|y^n\sum_{j=0}^n w_j(y+\iota y^{-1})^j\right\|_2^2,
\end{equation}
where $\|\cdot\|_2$ denotes the Euclidean norm of the coefficient vector.

Let $\widetilde X$ and $\widetilde Y$ be the values of the linear
functionals $X$ and $Y$, with $b=\iota$, applied to $\widetilde f$:
\[
  \widetilde X=\sum_{j=0}^{\floor{n/2}}w_{2j}(4\iota)^j,
  \qquad
  \widetilde Y=\sum_{j=0}^{\ceil{n/2}-1}w_{2j+1}(4\iota)^j;
\]
observe that
\[
  \widetilde X=|b|^{n/2}X,
  \qquad
  \widetilde Y=|b|^{(n+1)/2}Y.
\]
Since $Q_{n,\iota}(\mathbf w)=\Ht_\iota(\widetilde f)^2$, the fiber-volume
function for the unit ellipsoid $Q_{n,\iota}\leq1$ with respect to
$\widetilde X,\widetilde Y$ is precisely the previously defined
$\psi_{n,\iota}$.  Moreover, in the fiber coordinates $c_2,\ldots,c_n$ the
change from $w_2,\ldots,w_n$ back to $c_2,\ldots,c_n$ has Jacobian
\[
  \prod_{j=2}^n |b|^{-(n+j)/2}
  =|b|^{-(3n+2)(n-1)/4}.
\]
Hence
\[
  \psi_{n,b}(X,Y)
  =|b|^{-(3n+2)(n-1)/4}
  \psi_{n,\iota}
  \bigl(|b|^{n/2}X,|b|^{(n+1)/2}Y\bigr).
\]
Finally, in \eqref{eq:arch_factor} rescale the planar variables by
\[
  s\mapsto |b|^{-n/4}s,
  \qquad
  t\mapsto |b|^{-(n+2)/4}t.
\]
The Jacobian contributes $|b|^{-(n+1)/2}$, while the two arguments of
$\psi_{n,\iota}$ become $s^2+\iota t^2$ and $st$, respectively.
Therefore
\begin{equation}\label{eq:arch_factor_scale}
  \mathcal A_{n,b}
  =|b|^{-n(3n+1)/4}
  \int_{\RR^2}\psi_{n,\iota}(s^2+\iota t^2,st)\,ds\,dt.
\end{equation}

For the normalized integral in \eqref{eq:arch_factor_scale}, write
\[
  \widetilde f(y)=\sum_{i=0}^{2n}d_i y^i.
\]
Since $\widetilde f$ is $\iota$-reciprocal, $d_0,\ldots,d_n$ are free
coordinates, and
\[
  \Ht_\iota(\widetilde f)^2=d_n^2+2\sum_{i=0}^{n-1}d_i^2.
\]
The change from $w_0,\ldots,w_n$ to these free coefficients has
absolute determinant $1$.  Moreover,
\[
  \frac{\widetilde f(\sqrt \iota)}{(\sqrt \iota)^n}
  =\widetilde X+2\sqrt \iota\,\widetilde Y.
\]
Using the $\iota$-reciprocity relation to pair coefficients equidistant
from the center therefore gives
\begin{equation}\label{eq:XY_free_coefficients}
  \widetilde X=d_n+2\sum_{j=1}^{\floor{n/2}}\iota^j d_{n-2j},
  \qquad
  \widetilde Y=\sum_{j=0}^{\ceil{n/2}-1}\iota^{j+1}d_{n-2j-1}.
\end{equation}

We can now compute $F_{n,\iota}$ directly as the volume of an ellipsoid.
Fix $(s,t)\ne(0,0)$ and put
\[
  U=s^2+\iota t^2,\qquad V=st.
\]
On the $(\widetilde X,\widetilde Y)$-plane make the determinant-one
change of coordinates
\[
  \lambda=\frac{U\widetilde X+V\widetilde Y}{U^2+V^2},
  \qquad
  \mu=U\widetilde Y-V\widetilde X.
\]
Thus the line $(\widetilde X,\widetilde Y)=\lambda(U,V)$ is exactly
$\mu=0$, and the definition of $\psi_{n,\iota}$ shows that
\[
  F_{n,\iota}(s,t)=\int_{\RR}\psi_{n,\iota}(\lambda U,\lambda V)\,d\lambda
\]
is the $n$-dimensional volume of the section $\mu=0$ of the unit-height
ellipsoid, measured in the coordinates $\lambda$ and the fiber
coordinates.

Now set
\[
  z_n=d_n,\qquad z_i=\sqrt2\,d_i\quad(0\le i<n).
\]
The height ellipsoid becomes the unit ball $\sum z_i^2\le1$, while the
volume element contributes a factor $2^{-n/2}$.  By
\eqref{eq:XY_free_coefficients}, the coefficient vectors of
$\widetilde X$ and $\widetilde Y$ in the $z_i$-coordinates are
orthogonal and have squared lengths
\[
  2\floor{n/2}+1
  \qquad\text{and}\qquad
  \frac{\ceil{n/2}}{2},
\]
respectively.  Hence the coefficient vector of
$\mu=U\widetilde Y-V\widetilde X$ has squared length
\[
  D^2=\frac{\ceil{n/2}}{2}U^2+(2\floor{n/2}+1)V^2.
\]
After an orthogonal change of the $z_i$, we may take $\mu=D\xi_0$.
The section $\mu=0$ is then an ordinary unit $n$-ball, of volume $V_n$.
Thus its volume in the original coordinates is $V_n$ times the product
of the scaling factors $2^{-n/2}$ and $D^{-1}$, yielding
\begin{equation}\label{eq:F_direct}
  F_{n,\iota}(s,t)
  =\frac{V_n}{2^{(n-1)/2}
    \sqrt{\ceil{n/2}(s^2+\iota t^2)^2
    +2(2\floor{n/2}+1)s^2t^2}}.
\end{equation}

Using \eqref{eq:arch_factor_polar} with $b=\iota$, we obtain
\begin{align*}
  \int_{\RR^2}\psi_{n,\iota}(s^2+\iota t^2,st)\,ds\,dt
  &=\frac{V_n}{2^{(n+3)/2}\sqrt{\ceil{n/2}}}
    \int_0^{2\pi}
    \frac{d\theta}{\sqrt{
    (\cos^2\theta+\iota\sin^2\theta)^2
    +\dfrac{4\floor{n/2}+2}{\ceil{n/2}}
      \cos^2\theta\sin^2\theta}} \\
  &=\frac{V_n}{2^{(n+3)/2}\sqrt{\ceil{n/2}}}
    \int_0^{2\pi}
    \frac{d\theta}{\sqrt{1-
    \left(\dfrac{1-\iota}{2}
      -\dfrac{2\floor{n/2}+1}{2\ceil{n/2}}\right)
    \sin^2(2\theta)}} \\
  &=\frac{V_n}{2^{(n-1)/2}\sqrt{\ceil{n/2}}}
    K\left(\frac{1-\iota}{2}
      -\frac{2\floor{n/2}+1}{2\ceil{n/2}}\right).
\end{align*}
Combining this with
\eqref{eq:arch_factor_scale} gives
\begin{equation}\label{eq:arch_factor_value}
  \mathcal A_{n,b}
  =\frac{V_n}{2^{(n-1)/2}\sqrt{\ceil{n/2}}}
    K\left(\frac{1-\iota}{2}
    -\frac{2\floor{n/2}+1}{2\ceil{n/2}}\right)
    |b|^{-n(3n+1)/4}.
\end{equation}

It remains to compute the arithmetic factor $\<\Phi\>_{|b|}$, which is the mean value of $\gcd(s,b)$ among primitive pairs.  For a prime
$p^e\parallel b$ and a primitive pair $(s,t)$,
\[
  v_p(\gcd(s,b))=\min\{v_p(s),e\}.
\]
Conditioned on primitivity at $p$, for every $j\geq1$ we have
\[
  \Pr(p^j\mid s)=\frac{p^{1-j}}{p+1}.
\]
Since
\[
  p^{\min\{v_p(s),e\}}
  =1+(p-1)\sum_{j=1}^e p^{j-1}\mathbf 1_{p^j\mid s},
\]
its local mean is
\[
  1+(p-1)\sum_{j=1}^e p^{j-1}\frac{p^{1-j}}{p+1}
  =1+e\frac{p-1}{p+1}.
\]
The local conditions at the distinct primes dividing $b$ are
independent by the Chinese remainder theorem.  Hence
\begin{equation}\label{eq:gcd_mean}
  \<\Phi\>_{|b|}
  =\prod_{p^e\parallel |b|}\left(1+e\frac{p-1}{p+1}\right) = \rho(|b|).
\end{equation}

We can now return to the actual set $\mathcal S_H$ in
\eqref{eq:conic_reduced}.  Since $F_{n,b}\geq0$, the sandwich
\eqref{eq:SH_sandwich} together with \eqref{eq:weighted_primitive_sum_A} shows that
replacing $\mathcal S_H$ by the primitive pairs with
$0<\|(s,t)\|\leq\sqrt H$ affects our estimate of $\E_{n,b}$ by only
$O_{n,b}(H^n)$: the difference is bounded by the contribution of the
fixed-ratio annulus
$c\sqrt H<\|(s,t)\|\leq C\sqrt H$.  Hence
\eqref{eq:weighted_primitive_sum_A} with $R=\sqrt H$, together with
\eqref{eq:gcd_mean} and \eqref{eq:conic_reduced}, gives
\begin{align*}
  \#\{f:\Ht_b(f)\leq H,\ X^2-4bY^2\text{ is a square}\}
  &=\frac{\rho(|b|)}{2\zeta(2)}\mathcal A_{n,b}H^n\log H
    +O_{n,b}(H^n)\\
  &=\kappa_{n,\iota}\rho(|b|)|b|^{-n(3n+1)/4}H^n\log H
    +O_{n,b}(H^n),
\end{align*}
by \eqref{eq:arch_factor_value} and \eqref{eq:kappa_nonmonic}.
This proves
\eqref{eq:G1_nonmonic_asymptotic}.  The conditions that $g$ have degree $n$, be separable, and have $G_g=S_n$ remove only $O_{n,b}(H^n)$ polynomials, by \cite{Bhargava_vdW} and the arguments of \cite{ABO_RecipPolys}.

For the monic count, set $c_n=1$.  After dilation by $H$, the
corresponding affine section tends, in the normalized free coefficient
coordinates above, to $d_0=0$.  The height ball on this section is
\[
  d_n^2+2\sum_{i=1}^{n-1}d_i^2\leq1.
\]
The same direct calculation applies, now in dimension $n$: after the
$\sqrt2$-rescaling there are $n-1$ scaled coordinates $z_i$, and the squared
lengths of the coefficient vectors of
$\widetilde X$, $\widetilde Y$ in these coordinates become
\[
  2\ceil{n/2}-1
  \qquad\text{and}\qquad
  \frac12\floor{n/2}.
\]
These are exactly the two values appearing in the preceding calculation with
$n$ replaced by $n-1$.  Thus the real-place constant is $\kappa_{n-1,\iota}$, exactly as in \eqref{eq:kappa_nonmonic} with $n$ replaced by $n-1$.
The same argument, now with $n$ free
coefficients, yields the power $|b|^{-3n(n-1)/4}$, the main term
$H^{n-1}\log H$, and an error $O_{n,b}(H^{n-1})$.
\end{proof}

\subsection{Counting \texorpdfstring{$G_2$}{G2}-polynomials}\label{sec:G2}
In \cite{ABO_RecipPolys}, the most difficult group to deal with was $G_2$. We find here that the method works without essential change, so we briefly explain the adaptations needed to prove the following:
\begin{theorem} \label{thm:G2}
  For fixed $b\neq0$ and $n \geq 4$,
  \begin{align}
    \E_{n,b}(G_2; H) &\ll_{n,b} H^{n}  \label{eq:G2_nonmonic} \\
    \E_{n,b}^\monic(G_2; H) &\ll_{n,b} H^{n - 1}. \label{eq:G2_monic}
  \end{align}
\end{theorem}

By Lemma \ref{lem:G123_conds}\ref{it:G2}, we wish to count $g$ such that $g\bigl(2\sqrt{b}\bigr) g\bigl(-2\sqrt{b}\bigr) \disc g$ is a square. Let $\delta$ be a small constant, such as $1/(4n)$. We first prove the following analogue of Lemma 5.7 of \cite{ABO_RecipPolys}:
\begin{lemma} \label{lem:case1}
  Let $D$ be a positive integer. Let $C = \prod_{p\mid D} p$ be its radical; let $D'^2 = \prod_{p\mid D} p^{2\ceil{v_p(D)/2}}$ be its smallest square multiple. Assume that $C < H^{1 + \delta}$. The number of $b$-reciprocal integer polynomials $f$ of height $\leq H$ for which $G_f \subseteq G_2$ and $D \mid \Disc K_g$ is
  \[
  \ll \frac{O(1)^{\omega(C)} H^{n+1}}{D'^2}.
  \]
\end{lemma}
\begin{proof}
First we divide out from $D$ all primes $p \leq n$ and all primes $p \mid 2b$. If there is at least one $K_g$ with $D \mid \Disc K_g$, this change only affects $D$ by a bounded factor, since $v_p(\Disc K_g)$ is uniformly bounded.

For each prime $p \mid C$, excluding the degenerate case that $g \equiv 0$ mod $p$ (which is dealt with in the same way as in \cite{ABO_RecipPolys}), let $\sigma_p = (f_1^{e_1} \cdots f_r^{e_r})$ be the splitting type of the homogenization $\tilde g = y^n g(x/y)$. By \cite[Lemma 5.3, first inequality]{ABO_RecipPolys},
\[
  \ind(\tilde g \bmod p) \geq v_p(\Disc K_g) \geq v_p(D).
\]
Let
\[
  k_p=v_p(D')=\left\lceil\frac{v_p(D)}2\right\rceil.
\]
If $v_p(\Disc K_g)\geq 2k_p$, then
\[
  \ind(\tilde g \bmod p)\geq 2k_p,
\]
which gives the unpointed case $(0)$ below.

Otherwise, necessarily
\[
  v_p(\Disc K_g)=v_p(D)=2k_p-1,
\]
and therefore
\[
  v_p(\disc g)=v_p(\Disc K_g)
  +2v_p([\OO_{K_g}:\ZZ[\beta]])\equiv1\pmod2.
\]
Since $G_f\subseteq G_2$, Lemma \ref{lem:G123_conds}\ref{it:G2} implies that
\[
  g\bigl(2\sqrt b\bigr)g\bigl(-2\sqrt b\bigr)\disc g
\]
is a square, and hence
\[
  v_p\Bigl(g\bigl(2\sqrt b\bigr)g\bigl(-2\sqrt b\bigr)\Bigr)
  \equiv1\pmod2.
\]
In particular, $b$ must be a square modulo $p$: if not, then $p$ is inert in $\QQ(\sqrt b)$ (recall that $p\nmid 2b$), and the norm
\[
  g\bigl(2\sqrt b\bigr)g\bigl(-2\sqrt b\bigr)
  =N_{\QQ(\sqrt b)/\QQ}\bigl(g(2\sqrt b)\bigr)
\]
has even $p$-adic valuation, a contradiction. Thus choose $r_p\in\FF_p$ with $r_p^2=b$. Since the product
\[
  g(2r_p)g(-2r_p)
\]
vanishes modulo $p$, at least one of $u-2r_p$ and $u+2r_p$ divides $g(u)$ modulo $p$. We are therefore in one of the following three cases:
\begin{enumerate}[$($a$)$]
  \item[$(0)$] $\ind(\sigma_p)\geq2k_p$;
  \item[$(2)$] $\ind(\sigma_p)\geq2k_p-1$ and $f_1=1$ with $P_1=x-2r_py$;
  \item[$(-2)$] $\ind(\sigma_p)\geq2k_p-1$ and $f_1=1$ with $P_1=x+2r_py$.
\end{enumerate}
Accordingly, an annotated splitting type is again a pair $(\sigma,j)$ with
$j\in\{0,2,-2\}$. In the pointed cases $j = \pm 2$, the weighting
\[
  w_{p,\sigma,j}(h)=w'_{p,\sigma}\bigl(x^n h(y/x-jr_p)\bigr)
\]
picks out forms with a root at $jr_p$. The remainder of the proof proceeds
exactly as in \cite[\textsection 5.3]{ABO_RecipPolys}.
\end{proof}
As explained in the discussion following the statement of Lemma 5.7 in \cite{ABO_RecipPolys}, this lemma immediately establishes Theorem \ref{thm:G2} in Case I (following Bhargava's numbering of the cases), where $D \coloneqq \size{\Disc K_g} \geq H^{2 + 2\delta}$ and $C \coloneqq \prod_{p \mid D} p \leq H^{1 + \delta}$. The proof of Case II is unchanged from \cite[\textsection 5.4]{ABO_RecipPolys}, which is in turn unchanged from \cite[\textsection 5]{Bhargava_vdW}, since the (twisted) reciprocality of $f$ is not needed. Finally, for Case III, we construct a suitable analogue of the double discriminant by setting
\[
  h(c_0,\ldots, c_n) = g\bigl(2\sqrt{b}\bigr) g\bigl(-2\sqrt{b}\bigr) \disc g
\]
and
\begin{align*}
  R(c_0,\ldots, c_n) &= \Res_{c_0} \left(h, \frac{\partial}{\partial c_0} h\right) \\
  &= (-1)^{n(n-1)/2} c_n \disc_{c_0} h \\
  &= (-1)^{n(n-1)/2} c_n \disc_{c_0} \disc_u g \cdot \Bigl(g\bigl(2\sqrt{b}\bigr) - g\bigl(-2\sqrt{b}\bigr)\Bigr)^2 \cross \\
  &\quad {}\cross \Bigl(\disc_u \bigl(g - g\bigl(2\sqrt{b}\bigr)\bigr)\Bigr)^2 \Bigl(\disc_u \bigl(g - g\bigl(-2\sqrt{b}\bigr)\bigr)\Bigr)^2 
\end{align*}
(compare \cite[(13)]{ABO_RecipPolys}). Just as in \cite{ABO_RecipPolys}, we obtain that if $p \mid C$, then $p^2 \mid h$ for mod $p$ reasons, which implies $p \mid R$, and the rest of the proof proceeds as in \cite{ABO_RecipPolys}.

\subsection{Counting \texorpdfstring{$G_3$}{G3}-polynomials}\label{sec:G3}

Finally, we establish the following analogue of Theorem 6.1 of \cite{ABO_RecipPolys}.

\begin{theorem} \label{thm:G3}
  For fixed $b\neq0$ and $n \geq 3$ odd,
  \begin{align}
    \E_{n,b}(G_3; H) &\ll_{n,b} \begin{cases}
      H^2 \log^2 H & n = 3 \\
      H^{\frac{n+1}{2}} & n \geq 5
    \end{cases}  \label{eq:G3_nonmonic} \\
    \E_{n,b}^\monic(G_3; H) &\ll_{n,b} \begin{cases}
      H^2 & n = 3 \\
      H^2 \log H \log \log H & n = 5 \\
      H^{\frac{n-1}{2}} \log H & n \geq 7.
    \end{cases}
    \label{eq:G3_monic}
  \end{align}
\end{theorem}
\begin{proof}
First, it suffices to count $f$ that are primitive (i.e.\ of content $1$), since if we can establish $\E_{n,b}^{\prim}(G_3; H) \ll H^{\frac{n+1}{2}}$ for $n \geq 5$, then by summing over contents,
\[
  \E_{n,b}(G_3; H) = \sum_{k \geq 1} \E_{n,b}^\prim\Big(G_3; \frac{H}{k}\Big) \ll H^{\frac{n+1}{2}} \sum_{k \geq 1} \frac{1}{k^{\frac{n+1}{2}}} \ll H^{\frac{n+1}{2}},
\]
and likewise for all the other cases.

The Galois condition $G_f \subseteq G_3$ occurs when there is a factorization
\begin{equation} \label{eq:f_fax}
  f(x) = c \cdot h(x) \cdot x^n h\left(\frac{b}{x}\right)
\end{equation}
defined either over $\QQ$ or over a quadratic field $K_2 = \QQ\bigl(\sqrt{k}\bigr)$. In the former case, we may scale $h$ such that $\ct h = 1$. Then $|c^{-1}| = \ct(x^n h(b/x))$ is a divisor of $b^n$ and in particular is bounded. So $\Ht h = \Htp h \ll \sqrt{H}$, and we get at most $H^{(n+1)/2}$ polynomials $f$, as in the first part of \cite[\textsection 6.2]{ABO_RecipPolys}.

We are left with the case that \eqref{eq:f_fax} holds with two degree $n$ factors defined over $K_2 = \QQ\bigl(\sqrt{k}\bigr)$ for some squarefree integer $k$. Let $\aa = \ctidl(h)$ be the content ideal; by scaling, we may assume that $\aa$ is integral. We have
\[
  \ctidl\biggl( x^n h\Bigl(\frac{b}{x}\Bigr) \biggr) = \aa \bb,
\]
where $\bb$ is a divisor of $b^n$. Taking content ideals of both sides of \eqref{eq:f_fax}, we get
\[
  (1) = \ctidl(f) = (c)\aa^2 \bb.
\]
There are $O(1)$ possibilities for $\bb$, and for each $\bb$, the class of $\aa$ is unique up to multiplying by one of the $\size{\Cl(K_2)[2]} \asymp 2^{\omega(|k|)}$ ideal classes of order dividing $2$. Let us fix $\aa$ and $\bb$. 

Since $f \in \ZZ[x]$, there must be a constant $c' \in K_2^\cross$ such that
\begin{equation} \label{eq:h_conj}
  x^n h\left(\frac{b}{x}\right) = c' \bar{h}(x).
\end{equation}
Comparing contents of \eqref{eq:h_conj}, we get
\[
  (c') = \frac{\aa \bb}{\bar \aa},
\]
so in particular $|N_{K_2/\QQ}(c')| = N(\bb)$ is a divisor of $|b|^n$. If $K_2$ is imaginary, we get $\size{c'}_\infty \asymp 1$. If $K_2$ has two real places $v, \bar v$, note that the scaling $h \mapsto \eta h$ causes $c' \mapsto \pm \eta^2 c'$. By scaling by an appropriate power of the fundamental unit $\eta$, we can assume that $1 \leq \size{c'}_v < \size{\eta}_v^2$. Then there are $O(1)$ possible values of $c'$ given $\bb$ and $\aa$.

We have
\begin{align}
  \sqrt{H} &\geq \sqrt{\Ht f} \nonumber \\
  &= \sqrt{\Htp f} \nonumber \\
  &\asymp \Htp h \nonumber \\
  &= \prod_{v} \max\bigl\{\size{\theta_0}_v, \ldots, \size{\theta_n}_v\bigr\}^{[(K_2)_v : \QQ_v] / [K_2 : \QQ]} \nonumber \\
  &= \prod_{v \nmid \infty} N(\aa\OO_v)^{[(K_2)_v:\QQ_v]/2} \prod_{v\mid \infty} \max\bigl\{\size{\theta_0}_v, \ldots, \size{\theta_n}_v\bigr\}^{[(K_2)_v : \QQ_v] / 2} \nonumber \\
  &= N(\aa)^{-1/2} \prod_{v\mid \infty} \max\bigl\{\size{\theta_0}_v, \ldots, \size{\theta_n}_v\bigr\}^{[(K_2)_v : \QQ_v] / 2}. \label{eq:ht_bd}
\end{align}

In the case that $K_2$ is complex, this constrains the $\theta_i$ to lie in a disk $\D$ of area $O\big(N(\aa) H\big)$. Since $\aa$ has covolume $\Theta\big(\sqrt{\size{k}}\,N(\aa)\big)$, the number of lattice points of $\aa$ in $\D$ is $O(H/\sqrt{\size{k}})$, using that there are at least three noncollinear points in $\D \intsec \aa$ since $h$ is not a scalar multiple of an integer polynomial. Also, by \eqref{eq:h_conj}, the first $(n+1)/2$ coefficients $\theta_0,\ldots, \theta_{(n+1)/2}$ determine the others, so we get $O(H^{(n+1)/2} \size{k}^{-(n+1)/4})$ polynomials $h$ for given $\bb$ and $\aa$.

In the case that $K_2$ is real, the two factors in \eqref{eq:ht_bd} are linked by \eqref{eq:h_conj}, which tells us that
\[
  \max_i \size{\theta_i}_v \asymp \size{c'}_v \max_i \size{\theta_i}_{\bar v} .
\]
Hence
\[
  \size{\theta_i}_{v} \ll \sqrt{\size{c'}_v N(\aa) H} \textand \size{\theta_i}_{\bar v} \ll \sqrt{\frac{N(\aa) H}{\size{c'}_v}}.
\]
So the $\theta_i$ (in the Minkowski embedding) are constrained to lie in a rectangle $\D$ of area $O\big(N(\aa) H\big)$. As in the preceding case, this yields $O(H^{(n+1)/2} k^{-(n+1)/4})$ polynomials $h$ for given $\bb$ and $\aa$. In particular, we must have $k \ll H^2$ to get any $h$ at all. Lastly, we sum over $\bb$, $\aa$ and $k$ to get
\begin{align}
  \E_{n,b}(G_3; H) &\ll H^{(n+1)/2}\sum_{\substack{0<|k| \ll H^2\\\text{squarefree}}} \frac{2^{\omega(|k|)}}{|k|^{(n+1)/4}} \nonumber \\
  &\ll H^{(n+1)/2} \sum_{\substack{d,e \ll H^2}} \frac{1}{d^{(n+1)/4} e^{(n+1)/4}} \nonumber \\
  &\ll \begin{cases}
    H^{(n+1)/2} \log^2 H & n = 3 \\
    H^{(n+1)/2} & n \geq 5.
  \end{cases}
\end{align}
The monic case is similar and simpler, as we can take $\aa = (1)$, so $(c') = \bb$ and $\theta_n$ is a unit. There are $O(\log H/\log \size{k})$ units within the appropriate height bounds, yielding
\begin{equation*}
  \E_{n,b}^{\monic}(G_3; H) \ll \sum_{2 \leq |k| \ll H^2} \left(\frac{H}{\sqrt{|k|}}\right)^{\frac{n-1}{2}} \cdot \frac{\log H}{\log |k|}
\end{equation*}
which gives the claimed numerics as in \cite{ABO_RecipPolys}.
\end{proof}

\begin{proof}[Proof of Theorem \ref{thm:main}]
By Proposition \ref{prop:G_i}, after removing the polynomials with $G_g\neq S_n$, every exceptional polynomial is counted by $G_1$, $G_2$, or, when $n$ is odd, $G_3$.  Theorem \ref{thm:G1} gives the asserted main terms.  Theorem \ref{thm:G2} contributes $O_{n,b}(H^n)$ in the nonmonic case and $O_{n,b}(H^{n-1})$ in the monic case.  For $n\geq4$, the bounds in Theorem \ref{thm:G3} are smaller still.  The intersections among these subgroups are covered by the same error terms.  This proves both asymptotic formulas.

Finally, \eqref{eq:b_height_free} and the usual lattice-point estimate for an ellipsoid give
\[
  \#\{f:\deg f=2n,\ f\text{ is $b$-reciprocal},\ \Ht_b(f)\leq H\}
  =\frac{V_{n+1}}{2^{n/2}|b|^{3n(n+1)/4}}H^{n+1}
   +O_{n,b}(H^n),
\]
which yields the probability statement following Theorem \ref{thm:main}.
\end{proof}

\subsection{Uniformity in the twisting parameter}

We finish by recording what the preceding proofs give when $b$ is allowed to vary.  This range is not expected to be optimal.  Put $\iota=\sgn(b)$ and
\[
  \widetilde H=H/|b|^{n/2}.
\]
Equation \eqref{eq:height_Q} implies
\[
  |c_j|\ll_n \widetilde H|b|^{-j/2}\leq \widetilde H.
\]
We must have $\widetilde H \gg |b|^{n/2}$ to get any degree $2n$ polynomials at all. The elementary error terms in the $G_1$ calculation are absorbed by
\[
  O_{n,\epsilon}(|b|^{2n-2+\epsilon}\widetilde H^n).
\]
The same bound applies to the $G_2$ calculation.  Away from the primes dividing $2b$, its proof is uniform, while the discriminant of a degree-$n$ local field has valuation at most $n-1$ at primes exceeding $n$.  Removing the primes dividing $2b$ therefore loses at most $|b|^{2(n-1)}$ in the square-divisibility estimate; the finitely many primes at most $n$ contribute only a constant depending on $n$.  The occurrences of $b$ in the double discriminant contribute $|b|^\epsilon$ through the divisor bound.

Next consider the $G_3$ contribution, which occurs only when $n$ is odd.  The rational-factor part of the proof of Theorem \ref{thm:G3} gives an explicit exponent for the dependence on $|b|$.  As there, first take $f$ primitive and write
\[
  f(x)=c\,h(x)q(x),\qquad
  h(x)=\sum_{i=0}^n\theta_i x^i,\qquad
  q(x)=x^nh(b/x)=\sum_{i=0}^n\theta_i b^i x^{n-i},
\]
with $h\in\ZZ[x]$ primitive.  Since $\ctidl(q)\mid |b|^n$, we have
\[
  \Htp(q)=\frac{\max_i|\theta_i||b|^i}{\ctidl(q)}
  \geq |b|^{-n}\Htp(h).
\]
Multiplicativity of projective height up to constants depending only on $n$ therefore gives
\[
  H\geq\Ht(f)=\Htp(f)
  \gg_n\Htp(h)\Htp(q)
  \geq |b|^{-n}\Htp(h)^2,
\]
and hence
\[
  \Htp(h)\ll_n |b|^{n/2}H^{1/2}.
\]
Counting the $n+1$ integral coefficients of $h$, and then summing over the content of $f$ as in the proof of Theorem \ref{thm:G3}, bounds this contribution by
\[
  O_n\bigl(|b|^{n(n+1)/2}H^{(n+1)/2}\bigr).
\]
Since $H=|b|^{n/2}\widetilde H$, the factor in this bound is
\[
  |b|^{n(n+1)/2}H^{(n+1)/2}
  =|b|^{n(n+1)/2+n(n+1)/4}\widetilde H^{(n+1)/2}
  =|b|^{3n(n+1)/4}\widetilde H^{(n+1)/2}.
\]
In the genuinely quadratic-factor case, \eqref{eq:h_conj} gives $\Htp(q)=\Htp(\bar h)=\Htp(h)$.  The remaining dependence on $b$ in that argument comes from the content ideals dividing $(b^n)$ and the associated divisor factors, and is absorbed by the same power of $|b|$ in the estimates below.  Thus the catch-all exponent may be taken to be $3n(n+1)/4$.  We use this exponent in the monic estimate as well, without attempting to optimize it there.

Finally, the contribution from $G_g\neq S_n$ is $O_n(\widetilde H^n)$ by \cite{Bhargava_vdW} and \cite{ABO_RecipPolys}.  Combining these contributions, we obtain, for every $\epsilon>0$,
\begin{multline}\label{eq:uniform_nonmonic}
  \E_{n,b}(H)
  =\kappa_{n,\iota}\rho(|b|)|b|^{-n(3n+1)/4}H^n\log H\\
  {}+O_{n,\epsilon}\left(
    |b|^{2n-2-n^2/2+\epsilon}H^n+
    |b|^{n(n+1)/2}H^{(n+1)/2}\log^2\left(2H/|b|^{n/2}\right)
  \right),
\end{multline}
and
\begin{multline}\label{eq:uniform_monic}
  \E_{n,b}^{\monic}(H)
  =\kappa_{n-1,\iota}\rho(|b|)
    |b|^{-3n(n-1)/4}H^{n-1}\log H\\
  {}+O_{n,\epsilon}\left(
    |b|^{2n-2-n(n-1)/2+\epsilon}H^{n-1}+
    |b|^{n(n+2)/2}H^{(n-1)/2}\log^2\left(2H/|b|^{n/2}\right)
  \right).
\end{multline}
Here the second error term in each estimate may be omitted when $n$ is even, since $G_3$ does not then occur.

Since
\[
  1\leq\rho(|b|)\leq\prod_{p^e\parallel |b|}(e+1)=\tau(|b|)\ll_\epsilon |b|^\epsilon,
\]
the first error term in \eqref{eq:uniform_nonmonic} is smaller than the main term whenever
\begin{equation}\label{eq:uniform_range_nonmonic}
  |b|^{n(n+1)/4+2n-2+\epsilon}=o(\log H).
\end{equation}
The corresponding sufficient condition in the monic case is
\begin{equation}\label{eq:uniform_range_monic}
  |b|^{(n-1)(n/4+2)+\epsilon}=o(\log H).
\end{equation}
These ranges are not expected to be optimal, but a range exponential in some power of $|b|$ is necessary: on the hyperplane $Y = 0$, there are $\asymp |b|^{-n(3n+1)/4+1/2}H^n
\left(1+O\bigl(\frac{|b|^n}{H}\bigr)\right)$ polynomials, all of which automatically satisfy the $G_1$ condition, and these overshadow the main term unless
\[
  \log H\gg\frac{\sqrt{|b|}}{\rho(|b|)} \gg_\epsilon |b|^{1/2 - \epsilon}.
\]

\appendix
\section{Weighted primitive lattice points}\label{app:primitive_lattice}

\begin{proof}[Proof of Lemma~\ref{lem:weighted_primitive_lattice}]
It is enough to prove the assertion with $\Phi$ the characteristic
function of a single class $(u,v)\in P_q$, and then sum over the finitely
many classes.  Thus consider
\[
  S_{u,v}(T)=
  \sum_{\substack{(s,t)\in\ZZ^2,\ (s,t)=1\\
                   (s,t)\equiv(u,v)\pmod q}}
  W(s/T,t/T).
\]
By M\"obius inversion,
\[
  \mathbf 1_{(s,t)=1}=\sum_{d\mid s,\ d\mid t}\mu(d).
\]
Since $(u,v)\in P_q$, no prime dividing $q$ can divide both $s$ and $t$;
hence only $d$ with $(d,q)=1$ occur.  For such a $d$, the conditions
$d\mid s,t$ and $(s,t)\equiv(u,v)\pmod q$ specify, by the Chinese
remainder theorem, a single residue class modulo $qd$.  Standard Riemann
sum estimation on this translate of $(qd\ZZ)^2$ gives, uniformly in the
residue class,
\begin{equation}\label{eq:residue_riemann_sum}
  \sum_{(s,t)\equiv(\alpha,\gamma)\pmod{qd}}W(s/T,t/T)
  =\frac{T^2}{q^2d^2}\int_{\RR^2}W
   +O_{q,W}\left(\frac{T}{d}+1\right).
\end{equation}
The stated piecewise-$C^1$ hypotheses are more than enough for this
estimate: partition the support of $W$ into lattice squares of side
$qd/T$ and compare the value at the lattice point with the integral over
the square; squares meeting one of the finitely many boundary curves
contribute the same order of error.

Because $W$ is compactly supported, a nonzero pair contributing to the
sum has $d\ll_W T$.  Substituting \eqref{eq:residue_riemann_sum} into the
M\"obius sum therefore gives
\begin{align*}
  S_{u,v}(T)
  &=\frac{T^2}{q^2}\int_{\RR^2}W
    \sum_{\substack{d\geq1\\(d,q)=1}}\frac{\mu(d)}{d^2}
    +O_{q,W}(T\log(2T))\\
  &=\frac{T^2}{q^2}\int_{\RR^2}W
    \prod_{p\nmid q}(1-p^{-2})
    +O_{q,W}(T\log(2T)).
\end{align*}
(The omitted tail of the absolutely convergent $d^{-2}$ series contributes
$O_{q,W}(T)$.)  Finally,
\[
  |P_q|=q^2\prod_{p\mid q}(1-p^{-2}),
\]
so
\[
  \frac{1}{q^2}\prod_{p\nmid q}(1-p^{-2})
  =\frac{1}{\zeta(2)|P_q|}.
\]
Multiplying by $\Phi(u,v)$ and summing over $(u,v)\in P_q$ proves
\eqref{eq:weighted_primitive_lattice}.
\end{proof}

\bibliography{ourbib_4}
\bibliographystyle{alpha}
\end{document}